\documentclass[12pt]{amsart}
\usepackage{setspace}
\usepackage{tikz-cd}

\usepackage{geometry,amsxtra,amssymb,tikz-cd,amsfonts} 

\newtheorem{theorem}{Theorem}[section]

\newtheorem{rem}{Remark}[section]
\newtheorem{lem}{Lemma}[section]
\newtheorem{defn}{Definition}[section]

\newtheorem{conj}{Conjecture}[section]
\newcommand{\bz}{\mathbb{Z}}
\newcommand{\bg}{\mathbb{G}}

\newcommand{\bq}{\mathbb{Q}}

\newcommand{\bP}{\mathbf{Pic}}

\newcommand{\bd}{\mathbf{dim}}

\newcommand{\bF}{\mathbb{F}}
\newcommand{\bx}{\mathcal{X}}

\newcommand{\bl}{\mathcal{L}}
\newcommand{\by}{\mathcal{Y}}

\newcommand{\bs}{\mathbf{Spec}}

\newcommand{\mpp}{\mathcal{P}ic}
\newcommand{\mo}{\mathcal{O}}
\newcommand{\mg}{\mathcal{G}}
\newcommand{\mi}{\mathcal{I}}
\newcommand{\Br}{\mathbf{Br}}

\title{The period-index problem and the essential dimension of $\bg_m$ gerbes}
\author{Anningzhe Gao}
\date{} 

\begin{document}
\begin{abstract}
In this paper, we will consider the period-index problem of elliptic curves, and in particular we will give another concrete example of torsor with different period and index different with Cassel's example \cite{cassels1963arithmetic}. We will also define a number called 'generic index', which is closed related to the essential dimension of the Picard stack of algebraic curves with genus 1. Then we will show that the generic index is the same as the index.
\end{abstract}
\maketitle
\tableofcontents
\section{Introduction}
\subsection{Period-index problems}
Given an elliptic curve $E$ over some field $k$ with characteristic 0, let $C$ be a torsor of $E$ over $k$, then we know that $C$ can be regarded as an element in the first cohomology group $H^1(k,E)$. We may define the period of $C$ as the minimal positive integer $n$ such that $n[C]=0$ in $H^1(k,E)$. It can be shown that $n$ is always finite, denote by $per(C)$. We may also define the index of $C$, denote by $I(C)$, to be the greatest common divisor of the degrees of the closed points on $C$. Then a well known result is that $per(C)|I(C)$. A natural question is that whether these two numbers are the same. Unfortunately they are not. In \cite{cassels1963arithmetic}, Cassel first gave a counterexample with $per(C)=2$ but $I(C)=4$. On the other hand, we may define another value which lies between these two numbers: Given a torsor $C$ of $E$, we use $\mpp_{C/k}^0$ to denote its Picard stack of degree 0 component. Then $\mpp_{C/k}^0$ is a $\bg_m$ gerbe over $\bP_{C/k}^0\cong E$. Usually it is not a trivial gerbe when the field is not algebraically closed. We use $\bs\ K$ to denote the generic point of $E$, and consider the restriction of the Picard stack $\mpp_{C/k}^0$ to the generic point, say $\mpp_{C/k}^0|_K$. And we know the classes of $\bg_m$ gerbes over the field $K$ is the same as the set of similar classes of central simple algebras, i.e. the Brauer group $\Br(K)$, so we may consider the index of the Brauer class $\mpp_{C/k}^0|_K\in \Br(K)$. We call it the generic index of the torsor $C$, denoted by $i(C)$. We will see that these three numbers have the following relations:
$$per(C)|i(C)|I(C)|per(C)^2$$
It is natural to ask whether $per(C)=i(C)$ or $i(C)=I(C)$. In this paper, we will give positive answer to the second one:
\begin{theorem}\label{main}
Let $E/k$ be an elliptic curve over a field $k$ with characteristic 0. Then we always have $i(C)=I(C)$.

\end{theorem}

Here is the reason why we consider the number $i(C)$: It is closed related to the essential dimension (See definition below) of the stack $\mpp_{C/k}^0$. So the essential dimension of the stack $\mpp_{C/k}^0$ is closed related to the period-index problems.

\subsection{Essential dimension of algebraic stacks}
Roughly speaking, essential dimension is a number which measures the minimal parameters we need to describe some algebraic object. The theory of essential dimension of algebraic stacks has been widely studied, see \cite{Vistoli2007essential} and \cite{brosnan2009essential} for details. We review the definition of the essential dimension here. Fix some base field $k$, we use $Field/k$ to denote the category with objects the field extensions of $k$ and morphisms the obvious inclusions. Denote $Set$ the category of sets. Given a functor $$F:Field/k\to Set$$
Pick an element $\eta\in F(L)$ for some field $L/k$, we say an intermediate field $k\subseteqq L'\subseteqq L$ a defining field of $\eta$ if there exists some element $\eta'\in F(L')$ and $\eta'=\eta\in F(L)$ via the inclusion map. Then the essential dimension of $\eta$, denoted by $ed_k(\eta)$, is defined to be:
$$ed_k(\eta):=min_{L'}tr.deg(L'/k)$$
where $L'$ runs over the defining field of $\eta$ and $tr.deg$ means transcendental degree. The essential dimension of the functor $F$, is defined to be
$$ed_kF=max_{\eta\in F(L)}ed_k(\eta)$$
where $L$ runs over all field extensions and $\eta$ runs over all elements in $F(L)$. When the base field $k$ is clear, we will just write $ed(\eta)$ and $edF$.

From the definition it is obvious that the essential dimension of a general functor $F$ can be infinite. However, an interesting case is the following: consider an algebraic stack $\bx/k$, we can construct a functor
\begin{align*}
F_\bx:& Field/k\to Set\\
& L/k\to\{isomorphism\ classes\ in\ \bx(L) \}
\end{align*}
and define $ed_k\bx=ed_kF_\bx$. Of course for general algebraic stacks, even for moduli stacks, its essential dimension can be infinite, for example see Section 10 of \cite{Vistoli2007essential}. Usually we will pick $\bx$ to be the moduli stack of some moduli problems or $B_kG$, the classifying stack of some algebraic group $G$. We will review the theorems they proved we need to use later, especially the genericity theorem (Theorem 4.1 in \cite{Vistoli2007essential}).

In this paper we will consider the essential dimension of the Picard stack $\mpp_{C/k}^0$ defined as above. It is not a Deligne-Mumford stack so in general we cannot apply the genericity theorem. But we will see in the $\bg_m$ gerbe case we have a similar result. The second theorem we will prove here is:

\begin{theorem}\label{secondmain}
Assume Conjecture \ref{conj}. Let $C/k$ be an algebraic curve of genus 1 over some field $k$ with characteristic 0. We denote $\mpp_{C/k}^0$ the degree 0 Picard stack. Then we have $ed_k\mpp_{C/k}^0=1$ if and only $C$ admits a $k$ rational point.
\end{theorem}

\subsection{Outline of the paper} In Section 2 and Section 3 we will review the basic material of the period-index problem and essential dimension of algebraic stacks. We will prove a result we need to compute the essential dimension of Picard stacks in Section 3. 

In Section 4 we will consider the Picard stacks of algebraic curves of genus 1. We first give a well known theorem about the when the Picard stack is a trivial gerbe. Then we will concentrate on the Brauer group of elliptic curves. We establish a geometric interpretation of $\Br(E)$ for an elliptic curve $E$. Then we review the description of the 2-torsion elements in $\Br(E)$ given by Chernousov and Guletski in \cite{chernousov20012}.

In Section 5 we will discuss the number $I(C)$ for 2-torsion elements of $\Br(E)$ and in Section 6 we will give the proof of Theorem \ref{main}, also we will give another interesting example for $per(C)<I(C)$.

Through the whole paper, we always assume the character of the base field is 0. We will use the properties of Brauer groups and algebraic stacks freely. We refer to \cite{gille2017central} and \cite{33333} for the details of Brauer groups, \cite{laumon2018champs} and \cite{olssonstack} for the definition and properties of algebraic stacks and gerbes.

\subsection{Acknowledgment} The author would like to thank his advisor Martin Olsson introducing this interesting topic. The author also appreciate Gunter Harder, Max Lieblich, Martin Olsson and Minseon Shin for helpful discussions.

\section{Period-index problems of elliptic curves}

Fix the base field $k$. Given an elliptic curve $E/k$, and a torsor $C$ of $E$. We already defined the  period $per(C)$ and index $I(C)$ in the introduction section. Since $C$ is a trivial torsor if and only if $C$ admits a $k$ rational point, we can see that $I(C)$ can also defined as the greatest common divisor of the degrees of field extensions $l/k$ such that $[C]=0\in H^1(l,E)$ under the obvious inclusion. Usually for higher dimension abelian varieties, we don't have a closed point on $C$ with degree exact $I(C)$. But Artin and Lang showed that this is true in elliptic curves case, see Section 2 of \cite{lang1958principal}. That is, in our case, we have a closed point $p\in C$ with degree of its residue field just $I(C)$. It is shown in \cite{clark2004period}, Corollary 10 that we always have
$$per(C)|I(C)|per(C)^2$$
$per(C)$ and $I(C)$ are not always the same, see \cite{cassels1963arithmetic}.

We also give our definition of generic index of a torsor here:
\begin{defn}
\normalfont Given an elliptic curve $E$ over $k$. For a torsor $C$ of $E$, we define the \emph{generic index} of $C$ to be the index of
$$\mpp_{C/k}^0|_K\in \Br(K)$$
as a Brauer class over $K$, where $K$ is the functional field of $E$. We denote this number by $i(C)$.
\end{defn}

We summarize the basic facts of these three values in the following lemma.
\begin{lem}\label{basic}
The numbers $per(C),\ i(C),\ I(C)$ have the following relations:\\
$$per(C)|i(C)$$
$$i(C)|I(C)$$
\end{lem}
\begin{proof}
The inclusion $\Br(E)\to \Br(K)$ is injective (See \cite{milne1980etale}, Example 3.2.22), so we can see that $per(C)=per(\mpp_{C/k}^0|_K)$, hence we always have $per(C)|i(C)$. Also we know that suppose $C$ is a torsor splitting over some finite field extension $L$ with degree $d$, then we know $C$ admits a $L$ point, so $\mpp_{C_L/L}^0$ is a trivial gerbe over $E_L$, hence $K\otimes_kL$ is a splitting field of the Brauer class $\mpp_{C/k}^0|_K$, so $i(C)|L$, which implies $i(C)|I(C)$.
\end{proof}

\section{Essential dimension and algebraic stacks}

Given an algebraic stack $\bx$ over $k$, we define the essential dimension of the algebraic stack $ed\bx$ to be the essential dimension of the functor $F_\bx$ given by:
\begin{align*}
F_\bx:& Field/k\to Set\\
& L/k\to\{isomorphism\ classes\ in\ \bx(L) \}
\end{align*}

We will discuss the basic facts of essential dimension of algebraic stacks in this section. We refer to \cite{brosnan2009essential} and \cite{Vistoli2007essential} for more details. We first give some small lemmas.

\begin{lem}(\cite{Vistoli2007essential}, Prop. 2.12)
For an algebraic stack $\bx$ over $k$, suppose $L/k$ is a field extension, then we always have
$$ed_L\bx_L\leq ed_k\bx$$
\end{lem}

\begin{lem}(\cite{Vistoli2007essential}, Prop. 2.15)
Given an algebraic space $X$ over $k$ locally of finite type. We can define its essential dimension by considering it as a stack. Then we have
$$ed\bx=\bd X$$
\end{lem}

As we pointed in the introduction, usually the essential dimension of a general algebraic stack can be infinity. So we concentrate on some special algebraic stacks. We know that algebraic stacks come naturally as solutions of moduli problems, so moduli stacks are the things we will always focus on. In this paper we will consider two types: Deligne-Mumford stacks (DM stacks) and $\bg_m$ gerbes.

\subsection{The essential dimension of Deligne-Mumford stacks and genericity theorem}

Given an algebraic stack $\bx$ over $k$, recall the inertia stack $\mi_\bx\to \bx$ is the fiber product
$$\bx\times_{\bx\times\bx}\bx$$
mapping to $\bx$ via the second the projection. We say $\bx$ has finite inertia if $\mi_\bx\to\bx$ is finite. By the theorem proved by Keel and Mori, a DM stack $\bx$ over a field $k$ locally of finite type with finite inertia has a coarse moduli space $X$, and the morphism $\bx\to X$ is proper. In \cite{Vistoli2007essential}, Brosnan, Reichstein and Vistoli proved a theorem called $genericity\ theorem$ which is a powerful tool for computing essential dimension of some nice DM stacks. We now state it here:
\begin{theorem}\label{genericity}(\cite{Vistoli2007essential}, Theorem 4.1)
Let $k$ be a field of characteristic 0, $\bx$ a smooth DM stack with finite inertia, locally of finite type over $k$. Let $X$ be the coarse moduli space, $K$ the functional field of $X$. Denote $\bx_K$ the fiber product $\bx\times_X\bs K$, then we have
$$ed_k\bx=\bd X+ed_K\bx_K$$
\end{theorem}

This means that for DM stacks, to compute its essential dimension, we just need to consider the generic object. In lots of cases, this generic fiber $\bx_K$ will be some gerbe, so we will consider the essential dimension of gerbes next.

\begin{rem}
The Theorem \ref{genericity} can be generalized to DM stacks whose inertia is not finite, see \cite{brosnan2009essential} Section 6 for more details.
\end{rem}

\subsection{Essential dimension of gerbes}
We will focus on $\bg_m$ gerbes and $\mu_n$ gerbes in this section. Since we always assume the characteristic is 0, so $n$ is always invertible.

We fix a base field $k$. $\mu_n$ gerbes over $\bs k$ are classified by $H^2(k,\mu_n)$, $\bg_m$ gerbes are classified by $H^2(k,\bg_m)$. We have a natural injection
$$H^2(k,\mu_n)\to H^2(k,\bg_m)$$
under the language of Brauer groups, this is just the inclusion of $n$ torsion parts:
$$\Br(k)[n]\to \Br(k)$$
where for an abelian group $A$, we use $A[n]$ to denote it $n$ torsion part. The theory of Brauer groups of fields tells us $\Br(k)$ is a torsion group, that is $\Br(k)$ is the union of $\Br(k)[n]$ for $n$ runs over all positive integers. We will use this fact later.
 
Given $\mg\to \bs k$ a $\mu_n$ gerbe. $\mg$ corresponds to a Brauer class in $\Br(k)[n]$, let's use $\alpha$ to denote the class. Then we have the index of $\alpha$, we use $ind(\alpha)$ to denote it. It is a general philosophy that the essential dimension of a $\mu_n$ gerbe is closed related to the index of its Brauer class. In \cite{Canonical}, the authors conjectured the following: 
\begin{conj}\label{conj}(\cite{Canonical})
Given a $\mu_n$ gerbe $\mg$ over $\bs k$. If the corresponding Brauer class $\alpha$ has index $m$ with prime decomposition $m=p_1^{r_1}...p_k^{r_k}$, then we have
$$ed_k\mg=p_1^{r_1}+...+p_k^{r_k}-k+1$$
\end{conj}

This has been proved in \cite{Vistoli2007essential}, Theorem 5.4 when $k=1$. That is the case when $m$ is a prime power. Usually we always have
$$ed_k\mg\leq p_1^{r_1}+...+p_k^{r_k}-k+1$$

The case of $\bg_m$ gerbes is related to $\mu_n$ gerbes. 
\begin{theorem}\label{gmgerbe}(\cite{brosnan2009essential}, Theorem 4.1)

Suppose $\bx$ is a $\mu_n$ gerbe, $\by$ is the $\bg_m$ gerbe corresponding to $\bx$ under the inclusion
$$H^2(k,\mu_n)\to H^2(k,\bg_m)$$
Then we have
$$ed_k\bx=ed_k\by+1$$
\end{theorem}

\begin{lem}\label{tri}
Assume Conjecture \ref{conj}. Let $\mg\to\bs\ k$ be a $\bg_m$ gerbe over $\bs\ k$. Then $ed_k\mg=0$ if and only if $\mg$ is a trivial gerbe.
\end{lem}
\begin{proof}
Let $\alpha\in \Br(k)$ the Brauer class of $\mg$. Denote $I$ the index of $\alpha$. Suppose $I=p_1^{r_1}...p_{l}^{r_l}$ the prime decomposition of $I$. Assume Conjecture \ref{conj} is true. Then by Theorem \ref{gmgerbe}, we can see that
$$ed_k\mg=p_1^{r_1}+...+p_l^{r_l}-l=0$$
which can only happen if $I=1$, i.e. $\mg$ is a trivial gerbe. The converse is obvious.
\end{proof}

The genericity theorem also holds for $\bg_m$ gerbes. See for example \cite{reichstein2011genericity}, Lemma 2.4 and  we claim no originality. 
\begin{theorem}\label{genericity2}
Let $\bx\to X$ be a $\bg_m$ gerbe with $X$ an integral smooth variety over $k$, $\bx$ is locally of finite type. Denote $K$ the functional field of $X$, and $\bx_K$ the generic gerbe $\bx\times_X\bs K$. Then we have
$$ed_k\bx=\bd X+ed_K\bx_K$$
\end{theorem} 
\begin{proof}
Since $X$ is smooth, so $\Br(X)\to \Br(K)$ is an injection. So $\Br(X)$ is a torsion group. There exists a $\mu_n$ gerbe $\by$ over $X$ for some $n$ such that $\by$ maps to $\bx$ under the map
$$H^2(X,\mu_n)\to H^2(X,\bg_m)$$
$\by$ is a DM stack and we may apply the genericity theorem to see that
$$ed_k\by=\bd X+ed_K\by_K$$
where $\by_K$ is the generic gerbe. Then we have
\begin{align*}
ed_k\bx&=max_{p\in X}\{ed_{k(p)}\bx_p+tr.deg(k(p))\}\\
&=max_{p\in X}\{ed_{k(p)}\by_p+tr.deg(k(p))\}-1\\
&=\bd X+ed_K\by_K-1\\
&=\bd X+ed_K\bx_K
\end{align*}
where $k(p)$ is the residue field of $p\in X$ and $\bx_p$ is the restriction of $\bx$ to $p$, similar for $\by_p$. The first equality is by definition, the second is by Theorem \ref{gmgerbe}, the third is by Theorem \ref{genericity} and the remark below, the last is again by Theorem \ref{gmgerbe}.
\end{proof}

This theorem means for $\bg_m$ gerbes, we still have the genericity property.

We now actually proved:
\begin{theorem}\label{essentialofPicard}
Given an elliptic curve $E$ over $k$, $C$ a torsor of $E$. If $i(C)=p_1^{r_1}...p_k^{r_k}$, then we have
$$ed_k\mpp_{C/k}^0\leq p_1^{r_1}+p_2^{r_2}+...+p_k^{r_k}-k+1$$
and if conjecture \ref{conj} holds, we have the equality. 
\end{theorem}

Now Theorem \ref{secondmain} is just a corollary of Theorem \ref{essentialofPicard} and Lemma \ref{tri}.

This theorem gives us a reason to consider the value $i(C)$. It is natural to ask whether $i(C)$ is just $per(C)$ or $I(C)$? We will show that actually $i(C)=I(C)$. But first let's compute some concrete examples which generalize Cassel's result \cite{cassels1963arithmetic}.

\section{Picard stack of algebraic curves of genus 1 and Brauer group of elliptic curves}
 The main purpose of this section is to collect the tools we need in the construction in the next two sections. We first need:

\begin{theorem}\label{trivialgerbe}(\cite{harder39lectures}, Section 10.1.7)
Let $E/k$ be an elliptic curve over some field $k$ with characteristic 0 and $C$ is a torsor of $E$. Then $C$ is a trivial torsor if and only if $\mpp_{C/k}^0$ is a trivial $\bg_m$ gerbe over $E$.
\end{theorem}

Now we can prove Theorem \ref{secondmain}.

\begin{theorem}
Let $C/k$ be an algebraic curve of genus 1 over some field $k$ with characteristic 0. We denote $\mpp_{C/k}^0$ the degree 0 Picard stack. Then we have $ed_k\mpp_{C/k}^0=1$ if and only $C$ admits a $k$ rational point.
\end{theorem}
\begin{proof}
Let $E=\bP_{C/k}^0$ the degree 0 component of the Picard variety. By Theorem \ref{genericity2}, we can see that $ed_K\mpp_{C/k}^0|_K=0$, so by Theorem \ref{tri}, $[\mpp_{C/k}^0|_K]=0\in\Br(K)$. Since $\Br(E)\to\Br(K)$ is an injection (\cite{milne1980etale}, Example 3.2.22), so $\mpp_{C/k}^0\in\Br(E)$ is 0, hence $\mpp_{C/k}^0$ is a trivial $\bg_m$ gerbe. By Theorem \ref{trivialgerbe}, $C$ is a trivial $E$ torsor, so $C$ admits a $k$ rational point.
\end{proof}

\subsection{Canonical decomposition of $\Br(E)$}

We give the following geometric interpretation of the $\Br(E)$. The decomposition is well known but the author didn't find references for this theorem, so we give a proof here.

\begin{theorem}\label{decomp}
Given an elliptic curve $E$. We have a canonical decomposition 
$$\pi:\Br(k)\oplus H^1(k,E)\to \Br(E)$$
defined by for any $\mg\in \Br(k)$ and torsor $C$ of $E$
$$\pi(\mg,C)=f^*\mg+\mpp_{C/k}^0$$
where $f:E\to\bs k$ is the structure morphism.
\end{theorem}
\begin{proof}
This is a direct consequence of the Leray spectral sequence of the morphism $f:E\to\bs\ k$. We first show that $\pi$ is an injection. Suppose $\pi(f^*\mg+\mpp_{C/k}^0)=0\in \Br(E)$, then $f^*\mg+\mpp_{C/k}^0$ restricts to the identity of $E$ is a trivial gerbe, but the restriction of $\mpp_{C/k}^0$ to the identity is always trivial since we always have the structure sheaf, so $\mg=0\in \Br(k)$. Now $\mpp_{C/k}^0$ is a trivial gerbe over $E$ then by Theorem \ref{trivialgerbe} we must have $C$ is a trivial torsor. So $\pi$ is injective.

For surjectivity, we have the exact sequence
$$0\to \Br(k)\to \Br(E)\to H^1(k,E)\to 0$$
induced by the Leray spectral sequence, and the Picard stack $\mpp_{C/k}^0\in \Br(E)$ maps to $C\in H^1(k,E)$. So for any $\bx$ a $\bg_m$ gerbe over $E$, define $C\in H^1(k,E)$ to be the image of $\bx$. Then we have $\bx-\mpp_{C/k}^0$ maps to 0 under the morphism $\Br(E)\to H^1(k,E)$. So $\bx-\mpp_{C/k}^0=f^*\mg$ for some $\mg\in \Br(k)$. This proves the surjectivity.
\end{proof}

\subsection{2-torsion elements of $\Br(E)$}

We need to following useful description of the 2-torsion elements in $\Br(E)$ for an elliptic curve $E$ given in \cite{chernousov20012}. For any field $L$, two elements $a,b\in L^*$, we use the notation $<a,b>\in \Br(L)$ to denote the quaternion algebra generated by $1,i,j,ij$ with relations $$i^2=a,j^2=b,ij=-ji$$

We first set up the notations. Let $E/k$ be an elliptic curve over some field $k$ with characteristic 0, and suppose that the 2-torsion points of $E$ are defined over $k$. We use $\sigma,\tau,\omega$ the three non-trivial 2-torsion points of $E$, $e$ the identity point of $E$. We denote $f_{\sigma,\sigma}$ the rational function on $E$ with double zeroes at $\sigma$, double poles at $e$. Moreover, if we denote $\mo_{E,e}$ the local ring at the point $e$, and $\pi$ an uniformizer of it, then we need $\pi^2f_{\sigma,\sigma}\in\mo_{E,p}/\pi\mo_{E,p}$ a square in $k^*$. We can define $f_{\tau,\tau},f_{\omega,\omega}$ similarly. In the case when the elliptic curve is given by 
$$y^2=(x-a)(x-b)(x-c)$$
the three non-trivial 2-torsion points are $(a,0),(b,0),(c,0)$, and in this case we can set $f_{\sigma,\sigma}=x-a$, $f_{\tau,\tau}=x-b$, $f_{\omega,\omega}=x-c$.

\begin{theorem}\label{2torsion}(\cite{chernousov20012}, Theorem 3.6)
With the notations defined as above. All elements in $H^1(k,E)[2]\subseteqq \Br(E)[2]\subseteqq \Br(K)[2]$ can be written in the form:
$$<f_{\sigma,\sigma},r>\otimes<f_{\tau,\tau},s>$$
also all biquaternion algebras of this form arise from some torsors. And such a biquaternion algebra is trivial if and only if it is similar to one of the following three types:

(a) $<f_{\sigma,\sigma},u-b>\otimes<f_{\tau,\tau},u-a>$ where $u$ is the $x$ coordinate of some points in $E(k)$ with $u\neq a,u\neq b$.

(b) $<f_{\sigma,\sigma},a-b>\otimes<f_{\tau,\tau},(a-b)(a-c)>$

(c) $<f_{\sigma,\sigma},(b-a)(b-c)>\otimes<f_{\tau,\tau},b-a>$
\end{theorem}

With these tools, we can begin our discussion.

\section{A computation of $I(C)$ for 2-torsion elements in $\Br(E)$}

In this section we will discuss $I(C)$ in details. We will first fix our field to be $k$ of characteristic 0. We assume the elliptic curves we consider admits full 2-torsion points, that is $E[2]$ are all $k$ rational points. So the elliptic curve can be written as:
$$y^2=x(x-a)(x-b)$$
and $K$ its functional field, $e$ the identity point. We fix the following notations:
$$f_{\sigma,\sigma}=x-a$$
$$f_{\tau,\tau}=x-b$$

\subsection{The general theory about the case when $I(C)=2$}

By Theorem \ref{2torsion}, elements in $\Br(K)$ come from $\mpp_{C/k}^0$ with $per(C)=2$ if and only if it can be represented as:
$$<f_{\sigma,\sigma},M>\otimes<f_{\tau,\tau},N>$$
for some $A,B\in k$. Denote $C$ the torsor corresponding to this Brauer class. We will describe the case when $I(C)=2$.

We need some computation on elliptic curves. We suppose $I(C)=2$, and that means $C$ admits a closed points with degree 2, say $C(k(\sqrt{\alpha}))$ is not empty. Set $G_\alpha=Gal(k(\sqrt{\alpha})/k)$, and $g$ the only non-trivial element. So $[C]\in H^1(k,E)$ is represented by a 1 cocycle
\begin{align*}
\theta:& g\to p_g\\
& 1\to e
\end{align*}
for some point $p_g\in E(k(\sqrt{\alpha}))$ The case when $p_g$ is a 2-torsion point on $E$ is easy to control, so we assume $2p_g\neq e$. Since $\theta$ is a cocycle, we must have $p_g+gp_g=e$, so we must have $p_g=(A,\sqrt{A(A-a)(A-b)})$ and $\alpha/A(A-a)(A-b)$ is a square. We use $m$ to denote  $p_g/2$ (choose either one). By the standard calculation we have $m=(x_m,y_m)$ where
$$x_m=A+\sqrt{(A-a)(A-b)}+\sqrt{A(A-a)}+\sqrt{A(A-b)}$$
$$y_m=(x_m^2-ab)/2\sqrt{A}$$
so we can see that $L:=k(x_m,y_m)=k(\sqrt{A},\sqrt{A-a},\sqrt{A-b})$.  We have the following three cases, $[L:k]=2,4$ or 8. The first two cases are really similar to the last one, so we only discuss the case when $[L:k]=8$:

If $[L:k]=8$. Then we consider the following three elements in $Gal(L/k)$:
\begin{align*}
g:& \sqrt{A}\to-\sqrt{A}\\
&\sqrt{A-b}\to\sqrt{A-b}\\
&\sqrt{A-a}\to\sqrt{A-a}\\
\end{align*}
\begin{align*}
\beta:& \sqrt{A}\to-\sqrt{A}\\
&\sqrt{A-b}\to-\sqrt{A-b}\\
&\sqrt{A-a}\to\sqrt{A-a}\\
\end{align*}
\begin{align*}
\gamma:& \sqrt{A}\to-\sqrt{A}\\
&\sqrt{A-b}\to\sqrt{A-b}\\
&\sqrt{A-a}\to-\sqrt{A-a}\\
\end{align*}
Then $g,\beta,\gamma$ are generators of $Gal(L/k)$ and $\beta,\gamma$ fix the field $k(\sqrt{\alpha})$. By definition, the following two cocycles are same:
\begin{align*}
\eta_1:& g\to p_g\\
& \beta\to e\\
& \gamma\to e
\end{align*}
\begin{align*}
\eta_2:& g\to p_g+gm-m\\
& \beta\to e+\beta m- m\\
& \gamma\to e+\gamma m- m
\end{align*}
we can see that $\eta_1$ is just $\theta$ under the obvious restriction. Also $\eta_2$ can be regarded as elements in $H^1(k,E[2])$. Since the action of $Gal(\bar{k}/k)$ on $E[2]$ is trivial, $\eta_1$ is just a group homomorphism, and the kernel is generated by $g+\gamma+\beta$. Denote $F=k(\sqrt{A(A-a)},\sqrt{A(A-b)})$. Define $\mu,\nu\in Gal(F/k)$ where $\mu$ sends $\sqrt{A(A-a)}$ to $-\sqrt{A(A-a)}$ and $\nu$ sends $\sqrt{A(A-b)}$ to $-\sqrt{A(A-b)}$. We can see that $g+\gamma=\mu$ and $g+\beta=\nu$ in $Gal(k(\sqrt{A(A-a)},\sqrt{A(A-b)})/k)$. We have $e+(g+\gamma)m-m=\sigma$ and $e+(g+\beta)m-m=\tau$ by direct calculation. So define:
\begin{align*}
\eta:& \mu\to \sigma\\
& \nu\to \tau\\
\end{align*}
we can see that $\eta$ and $\gamma$ are the same if we restricts to $H^1(Gal(L/k),E)$. So $[C]\in H^1(k,E)$ is also represented by $\eta$. That means, the Brauer class 
$$<f_{x_\sigma,\sigma},M>\otimes<f_{\tau,\tau},N>$$
has index 2 if and only if it is isomorphic to
$$<f_{\sigma,\sigma},A(A-a)>\otimes<f_{\tau,\tau},A(A-b)>$$
and the splitting field is $k(\sqrt{A(A-a)(A-b)})$. The cases when $[L:k]=2,4$ are the same. From Mordell-Weil theorem we know that $E(k)/2E(k)$ is a finite set. By the exact sequence
$$E(k)/2E(k)\to H^1(k,E[2])\to H^1(k,E)[2]\to0$$
there is a finite set of pairs of integers $P=\{(\beta_1,\gamma_1),...,(\beta_n,\gamma_n)\}$ with elements in $k^*/(k^*)^2\times k^*/(k^*)^2$ such that 
$$<f_{\sigma,\sigma},M>\otimes<f_{\tau,\tau},N>\cong<f_{\sigma,\sigma},M'>\otimes<f_{\tau,\tau},N'>$$
if and only if $(MM',NN')\in P$ (Since everything is in $k^*/(k^*)^2\times k^*/(k^*)^2$ so $M/M'=MM'$). So we have
\begin{theorem}\label{2index}
Let $E/k$ be an elliptic curve over $k$ with characteristic 0. Given some element 
$$<f_{\sigma,\sigma},M>\otimes<f_{\tau,\tau},N>$$
coming from some torsor $C$ of $E$. Let $P=\{(\beta_1,\gamma_1),...,(\beta_n,\gamma_n)\}$ the set of pairs of integers coming from $E(k)/2E(k)$. Then $I(C)=2$ if and only for some $A\in k$, we have $(MA(A-a),NA(A-b))\in P$, or the Brauer class of $\mpp_{C/k}^0|_K$ is isomorphic to $<f_{\sigma,\sigma},A>$, $<f_{\tau,\tau},A>$ or $<f_{\omega,\omega},A>$ for some $A\in k$.
\end{theorem}

The theorem seems hard to control, but we will see in the next section that in specific cases it is really clear.

\subsection{Concrete examples of the elliptic curve $y^2=x(x^2-1)$}

In this part we concentrate on elliptic curve $E$ defined by
$$y^2=x(x^2-1)$$
over $k=\bq$. The discussion in this part can be easily generalized, the only reason we choose this elliptic curve is that the $k$ rational points of $E$ can be found. It is easy to see that $E[2]=E(k)$. So by Theorem \ref{2torsion},
$$<f_{\sigma,\sigma},M>\otimes<f_{\tau,\tau},N>$$
if and only if $(M,N)=(1,1),(1,-1),(2,2),(2,-2)$ in $k^*/(k^*)^2\times k^*/(k^*)^2$. Suppose we have some torsor $C$ with $\mpp_{C/k}^0|_K$ is represented by 
$$<f_{\sigma,\sigma},M>\otimes<f_{\tau,\tau},N>$$
Denote $P=\{(1,1),(1,-1),(2,2),(2,-2)\}$. Then from Theorem \ref{2index} we know that $I(C)=2$ if and only there exists some $A\in k$ such that
$(MA(A-1),N(A+1))\in P$. Let's $(MA(A-1),NA(A+1))=(1,1)$ in $k^*/(k^*)^2\times k^*/(k^*)^2$. Then we have equations:
$$A(A-1)=Mx^2$$
$$A(A+1)=Ny^2$$
for some $x,y\in k$. Take the sum we have $2A^2=Mx^2+Ny^2$. This is the same as $<2M,2N>=1\in \Br(k)$. On the other hand, suppose we have $<2M,2N>=1$, this means there exists some rational numbers $x,y\in k$ such that $Mx^2+Ny^2=2z^2$. Take $r=(Ny^2-Mx^2)/2$. Define 
$$A=\frac{s^2}{r}$$
we can see that $A$ satisfies equations:
$$A(A-1)=Mx^2s^2/r^2$$
$$A(A+1)=Ny^2s^2/r^2$$ 
so $(A(A-1),A(A+1))=(M,N)$ in $k^*/(k^*)^2\times k^*/(k^*)^2$. We can check other three cases similarly so we have proved:
\begin{theorem}\label{result}
Let $E/\bq$ be the elliptic curve defined by
$$y^2=x(x^2-1)$$
then the Brauer class
$$<f_{\sigma,\sigma},M>\otimes<f_{\tau,\tau},N>$$
has index 2 if and only if at least one of the following quaternion algebras
$$<M,N>,<M,-N>,<2M,2N>,<2M,-2N>$$
splits.
\end{theorem}
\begin{rem}
The theorem can be generalized to any elliptic curves directly. Set $E/k$ defined by $y^2=x(x-a)(x-b)$. If we denote $P=\{(\beta_1,\gamma_1),...,(\beta_n,\gamma_n)\}$ as usual, then we can see that 
$$<f_{\sigma,\sigma},M>\otimes<f_{\tau,\tau},N>$$
has index 2 if and only if one of the following quaternion algebras:
$$<-(a-b)bM\beta_i,(a-b)aN\gamma_i>$$
for $1\leq i\leq n$ is trivial.
\end{rem}
\begin{rem}
The integers $M,N$ are not symmetric since for example we have $<f_{\tau,\tau},-1>$ is trivial while $<f_{\sigma,\sigma},-1>$ is not.
\end{rem}

From the theorem we can give infinitely many torsors $C$ of $E$ with $per(C)=2$ but $I(C)=4$ in details. For example, we pick $(M,N)=(-1,7)$, then we can see in this case the index is 4. This generalizes Cassel's construction \cite{cassels1963arithmetic}.

\section{The proof of Theorem \ref{main}}

We prove the equality $i(C)=I(C)$ in this section. We set our notations: $E/k$ is an elliptic curve over $k$ with characteristic 0, $C$ is a torsor of $E$. $K$ is the function field of $E$.  $\mpp_{C/k}^0\to E^\vee$ is a $\bg_m$ gerbe (Although $E^\vee\cong E$ but we use the dual abelian variety to distinguish with the original one). Suppose $\mpp_{C/k}^0|_K$ splits over some finite field extension $L/K$ with degree $i(C)$, then we have a smooth curve $\pi:\Sigma\to E^\vee$ finite over $E^\vee$ of degree $i(C)$ such that the pull-back of the $\bg_m$ gerbe $\mpp_{C/k}^0\to E^\vee$ under the covering map $\Sigma\to E^\vee$ is trivial (See \cite{AG}, Section 1.6 for details, $\Sigma$ mat not be geometrically connected but this doesn't matter). So there is a line bundle $L$ on $\Sigma\times C$ such that for any point $x\in \Sigma$ we have $L|_{x\times C}\cong L_{\pi(x)}$, here $\pi(x)$ is a point on $E^\vee$ so $L_{\pi(x)}$ just means the corresponding line bundle on $C$. We have two projections $$p:\Sigma\times C\to \Sigma$$
$$q:\Sigma\times C\to C$$
We claim that $deg(Rq_*L)=-i(C)$. To prove the claim we may assume $k$ is algebraically closed. So we have $C\cong E$. We will use $E$ instead of $C$ in the further proof. We denote the Poincare line bundle on $E\times E^\vee$ by $P$, then we know that $L\cong p^*M\otimes (\pi\times1)^*P$ for some line bundle $M$ on $\Sigma$ (See \cite{Bhatt}, Chapter 6). Similar we have two projections:
$$p_1:E^\vee\times E\to E^\vee$$
$$q_1:E^\vee\times E\to E$$
then we have
$$Rq_*L\cong Rq_{1*}(\pi\times1)_*(p^*M\otimes(\pi\times1)^*P)\cong Rq_{1*}((\pi\times1)_*p^*M\otimes P)$$
By the flat base change theorem (\cite{Bhatt}, Chapter 4), we have $(\pi\times1)_*p^*M\cong p_1^*\pi_*M$, hence
$$Rq_*L\cong Rq_{1*}(p_1^*\pi_*M\otimes P)=\Phi_{E^\vee}(\pi_*M)$$
Here $\Phi_{E^\vee}$ means the Fourier-Mukai transform  $D^b(E^\vee)\to D^b(E)$ defined by the kernel $P$, see \cite{Bhatt}, Chapter 16,17. Then by \cite{Bhatt} Lemma 21.5 we have $$deg(\Phi_{E^\vee}(\pi_*M))=-rank(\pi_*M)=-i(C)$$, we proved the claim.

Now let $N=det(Rq_*L)^\vee$ be the dual of the determinant, we have $deg(N)=i(C)$. Then by definition we can see that $I(C)|i(C)$, combine with Lemma \ref{basic}, we get $$i(C)=I(C)$$

\subsection{Another example of $per(C)<I(C)$}

In this subsection we will construct another example with $per(C)<I(C)$.

We set $k=\bq(t_1,t_2,t_3,t_4)$. Here actually $\bq$ can be replaced by any field of char 0. We define an elliptic curve $E/k$ by
$$y^2=x(x-t_1)(x-t_2)$$
By Theorem \ref{2torsion}, the central simple algebra
$$A=<f_{\sigma,\sigma},t_3>\otimes<f_{\tau,\tau},t_4>$$
comes from some torsor $C$. We have $per(C)=2$. Now we have the following:
\begin{theorem}\label{construction2}
The central simple algebra $A$ has index 4 in $\Br(K)$.
\end{theorem}
\begin{proof}
By \cite{gille2017central} Theorem 1.5.5, $A$ has degree 4 if and only if the equation
$$f_{\sigma,\sigma}u^2+t_3v^2-t_3f_{\sigma,\sigma}w^2=f_{\tau,\tau}r^2+t_4s^2-t_4f_{\tau,\tau}p^2$$
has no non-trivial solutions. Now we have $f_{\sigma,\sigma}=x$, $f_{\tau,\tau}=x-t_1$, and we know $K\cong k(x)[y]/(y^2-x(x-t_1)(x-t_2))$, so every element in $K$ can be written in the form $fy+g$ where $f,g$ are rational functions of $x$. Then we write every element in the equation in the explicit form, the equation is the same as the following two equations:
$$xu_1u_2+t_3v_1v_2-t_3xw_1w_2=(x-t_1)r_1r_2+t_4s_1s_2-t_4(x-t_1)p_1p_2$$
$$x(u_1^2x(x-t_1)(x-t_2)+u^2_2)+t_3(v_1^2x(x-t_1)(x-t_2)+v_2^2)-t_3x(w_1^2x(x-t_1)(x-t_2)+w_2^2)$$
$$=(x-t_1)(r_1^2x(x-t_1)(x-t_2)+r_2^2)+t_4(s_1^2x(x-t_1)(x-t_2)+s_2^2)-t_4(x-t_1)(p_1^2x(x-t_1)(x-t_2)+p_2^2)$$
We may assume that all things appear in the equation are polynomials of $x$. We will use infinite descend to get a contradiction. For simplicity, we will use the same notations when we consider things modulo some element. In the following proof, we will concentrate on the second equation, cause the first one will be satisfied automatically. Suppose we have a non-trivial solution, we may assume that the sum of their degrees (as polynomials in $k[x]$) is minimal.

Let $x=0$, we have
$$t_3v_2^2=-t_1r_2^2+t_4s_2^2+t_4t_1p_2^2$$
Here $v_2,r_2,s_2,p_2$ means there value at $x=0$, same for the following discussion. We show that this equation has only trivial solution, in other words, we must have 
$$x|v_2,r_2,s_2,p_2$$
in the original equation.

Assume this is not true. Since $v_2,r_2,s_2,p_2$ are rational functions in $t_1,t_2,t_3,t_4$, we can regarded them as polynomials in $t_4$ and coefficients in $\bq(t_1,t_2,t_3)$. We may also assume not all of them are divided by $t_4$. If $t_4\nmid v_2\ or\ r_2$, then set $t_4=0$ we will see that $-t_1t_3$ will be a square in $\bq(t_1,t_2,t_3)$, which is not true. So $t_4|v_2,r_2$. Write $v_2=t_4v_2',r_2=t_4r_2'$, we have
$$t_3t_4v_2'^2=-t_1t_4r_2'^2+s_2^2+t_1p_2^2$$
By our assumption one of $s_2,p_2$ cannot be divided by $t_4$, this implies $t_1$ is a square in $\bq(t_1,t_2,t_3)$, which cannot happen. So we have $x|v_2,r_2,s_2,p_2$ in the original equation.

Write $v_2=xv_2',r_2=xr_2',s_2=xs_2'p_2=xp_2'$. We have
$$(u_1^2x(x-t_1)(x-t_2)+u^2_2)+t_3(v_1^2(x-t_1)(x-t_2)+xv_2'^2)-t_3(w_1^2x(x-t_1)(x-t_2)+w_2^2)$$
$$=(x-t_1)(r_1^2(x-t_1)(x-t_2)+xr_2'^2)+t_4(s_1^2(x-t_1)(x-t_2)+xs_2'^2)-t_4(x-t_1)(p_1^2(x-t_1)(x-t_2)+xp_2'^2)$$

We let $x=0$, then we have
$$u_2^2+t_3t_1t_2v_1^2-t_3w_2^2=-t_1^2t_2r_1^2+t_1t_2t_4s_1^2+t_1^2t_2t_4p_1^2$$
Same as before we will show that this equation will only have trivial solution, which means
$$x|u_2,v_1,w_2,r_1,s_1,p_1$$
in the original one. We can consider $u_2,v_1,w_2,r_1,s_1,p_1$ are polynomials in $t_4$ with coefficients in $\bq(t_1,t_2,t_3)$, and not all of them are divided by $t_4$. If one of $u_2,v_1,w_2,r_1$ is not divided by $t_4$, by letting $t_4=0$, we have
$$u_2^2+t_1t_2t_3v_1^2-t_3w_2^2=-t_1^2t_2r_1^2$$
where $u_2,v_1,w_2,r_1\in\bq(t_1,t_2,t_3)$ and not all of them are zeroes. Then we may consider them as polynomials in $t_3$ and coefficients in $\bq(t_1,t_2)$. Similar as before we may assume not all of are divided by $t_3$. If $t_3$ doesn't divide one of $u_2,r_1$, modulo $t_3$ will lead to $-t_2$ is a square in $\bq(t_1,t_2)$, which is a contradiction. So $t_3|u_2,r_1$. Divide $t_3$ and since $t_3$ doesn't divide one of $v_1,w_2$, this leads to $t_1t_2$ a square in $\bq(t_1,t_2)$, which is a contradiction. So we must have 
$$t_4|u_2,v_1,w_2,r_1$$
Divide $t_4$ since $t_4\nmid s_1$ or $t_4\nmid p_1$, this implies $t_1$ is a square in $\bq(t_1,t_2,t_3)$, which is a contradiction. So we must have 
$$x|u_2,v_1,w_2,r_1,s_1,p_1$$
Write $u_2=xu_2',v_1=xv_1',w_2=xw_2',r_1=xr_1',s_1=xs_1',p_1=xp_1'$, we have
$$(u_1^2(x-t_1)(x-t_2)+xu'^2_2)+t_3(v_1'^2x(x-t_1)(x-t_2)+v_2'^2)-t_3(w_1^2(x-t_1)(x-t_2)+xw_2'^2)$$
$$=(x-t_1)(r_1'^2x(x-t_1)(x-t_2)+r_2'^2)+t_4(s_1'^2x(x-t_1)(x-t_2)+s_2'^2)-t_4(x-t_1)(p_1'^2x(x-t_1)(x-t_2)+p_2'^2)$$
The following argument is really the same. We can conclude that $x|u_1,w_1$. Write $u_1=xu_1',w_1=xw_1'$, then 
$$u_1',u_2',v_1',v_2',w_1',w_2',r_1',r_2',s_1',s_2',p_1',p_2'$$
form a new solution of the original equation with smaller degree in $x$, which is a contradiction. So we can see that 
$$A=<f_{\sigma,\sigma},t_3>\otimes<f_{\tau,\tau},t_4>$$
is a division algebra, hence $i(C)=4$.
\end{proof}

For $C$ constructed in this section we have
$$ed_{k}\mpp_{C/k}^0=4$$

\bibliographystyle{abbrv}
\bibliography{MyCitation}

\end{document}